\documentclass[a4paper,10pt]{article}
\usepackage{mathtext}
\usepackage[T1,T2A]{fontenc}
\usepackage[cp1251]{inputenc}
\usepackage[english]{babel}
\usepackage{amsmath}
\usepackage{amsfonts}
\usepackage{amssymb}
\usepackage{mathrsfs}
\usepackage{amsthm}
\usepackage{euscript}
\DeclareMathOperator{\dist}{dist}

\DeclareMathOperator{\supp}{supp}

\DeclareMathOperator{\HS}{\mathcal{H}}
\DeclareMathOperator{\BMO}{\mathrm{BMO}}

\newtheorem{Def}{Definition}

\renewcommand{\leq}{\leqslant}
\renewcommand{\geq}{\geqslant}
\newcommand{\scalprod}[2]{\langle{#1},{#2}\rangle}

\theoremstyle{plain}\newtheorem{Th}{Theorem}
\theoremstyle{plain}\newtheorem{Le}{Lemma}
\theoremstyle{plain}\newtheorem{Cor}{Corollary}
\theoremstyle{plain}\newtheorem{St}{Proposition}

\begin{document}
\title{Dorronsoro's theorem and a small generalization.}
\author{D. M. Stolyarov\thanks{Supported by RSF grant 14-41-00010.}}

\maketitle
\begin{abstract}
We give a simple proof of Dorronsoro's theorem \textup(Theorem~$2$ in~\cite{Dorronsoro}\textup) and use similar ideas to establish an equivalence for embeddings of vector fields.
\end{abstract}
\section{Introduction}
\begin{Th}[Dorronsoro's theorem\textup, Theorem~$2$ in~\cite{Dorronsoro}]\label{DorronsoroTheorem}
For any real-valued function~$f \in C_0^{\infty}(\mathbb{R}^d)$\textup,~$d \geq 2$\textup, there exists a real-valued function~$F\in \HS_1$ such that
\begin{equation*}
I_1[F] \geq f;\quad \|F\|_{\HS_1} \lesssim \|\nabla f\|_{L_1}.
\end{equation*}
\end{Th}
We denote the space of all compactly supported smooth functions by~$C_0^{\infty}$, the real Hardy class by~$\HS_1$ (we address the reader to the book~\cite{Stein} where he can find all the material about the Hardy class~$\HS_1$ and the~$\BMO$ space) and the Riesz potential of order~$a$ by~$I_{a}$,
\begin{equation*}
I_{a}[f] = f*c_a|\cdot|^{a - d},\quad f \in C_0^{\infty}(\mathbb{R}^d),\; a \in (0,d).
\end{equation*}
Here~$c_a$ is the constant such that~$I_{a}$ is the Fourier multiplier with the symbol~$|\xi|^{-a}$. Surely, the Riesz potentials may be applied to a function belonging to~$\HS_1$. Here and in what follows,~``$a \lesssim b$'' means~``$a \leq cb$ for a uniform constant~$c$''. We also always assume~$d \geq 2$.

In the original formulation of Theorem~\ref{DorronsoroTheorem},~$f$ belongs to the homogeneous space~$\mathrm{BV}$ of functions of bounded variation (and the~$L_1$-norm of the gradient in the estimate is replaced by its total variation). This more general statement easily follows from Theorem~\ref{DorronsoroTheorem} by approximation. Though Theorem~\ref{DorronsoroTheorem} may seem a bit sophisticated, we give a corollary that emphasizes its importance.
\begin{Cor}\label{LorentzSobolevEmbedding}
$\dot{W}_1^1(\mathbb{R}^d) \hookrightarrow L_{\frac{d}{d-1},1}(\mathbb{R}^d)$.
\end{Cor}
Here~$\dot{W}_1^1$ is the homogeneous Sobolev space, which is the completion of the set~$C_0^{\infty}$ with respect to the norm
\begin{equation*}
\|f\|_{\dot{W}_1^1} = \|\nabla f\|_{L_1}.
\end{equation*}
In what follows, it is convenient to work with complex-valued functions also; we assume that a function in~$\dot{W}_1^1$ is complex-valued. The symbol~$L_{\frac{d}{d-1},1}$ denotes the Lorentz space (see the book~\cite{Grafakos} for a detailed study of these spaces). Corollary~\ref{LorentzSobolevEmbedding} was proved in~\cite{Faris}, however, see the paper~\cite{Kolyada} for even more general (with respect to another interpolation parameter) result. Corollary~\ref{LorentzSobolevEmbedding} follows from Theorem~\ref{DorronsoroTheorem} if one recalls that the Riesz potential~$I_1$ maps~$\HS_1$ to~$L_{\frac{d}{d-1},1}$ (this may be justified by means of real interpolation, see~\cite{FRS}; otherwise, use the atomic decomposition).

We give a proof of Theorem~\ref{DorronsoroTheorem} in the next section. It differs from the original proof in~\cite{Dorronsoro} by two points: it is constructive (i.e. the function~$F$ may be computed in terms of~$f$), the original proof used various duality arguments several times; the presented proof may seem more transparent, because we use only some basic geometric facts (such as Gustin's boxing inequality or the coarea formula) without going into detailed study of fractional maximal functions. However, the machinery that works in our proof is the same as in the original.

In Section~\ref{sGeneral}, we show that in a more general setting, the statements in the style of Theorem~\ref{DorronsoroTheorem} are equivalent to a proper analog of Gustin's inequality.

Finally, we collect the statements we use without proof in the last section.

The author is grateful to A. I. Nazarov and the anonymous referee for exposition advice and corrections.

\section{Proof of Theorem~\ref{DorronsoroTheorem}}
We begin with an easy lemma that lies in the heart of all our constructions. By~$(-\Delta)^{\frac12}$ we denote the Fourier multiplier with the symbol~$|\xi|$.
\begin{Le}\label{TestFunctionLemma}
For any function~$\varphi \in C_0^{\infty}(\mathbb{R}^d)$\textup, the function~$(-\Delta)^{\frac12}\varphi$ is in~$\HS_1$.
\end{Le}
\begin{proof}
We proceed in several steps. First, we show that~$(-\Delta)^{\frac12}\varphi \in L_{\infty}(\mathbb{R}^d)$. Indeed, the Fourier transform of this function belongs to~$L_1$, because it is bounded and decays rapidly at infinity.

Second, we show that~$(-\Delta)^{\frac12}\varphi(x) \lesssim (1+|x|)^{-d-1}$. Since~$(-\Delta)^{\frac12}\varphi \in L_{\infty}$, it suffices to verify the inequality only for~$x \notin \supp \varphi$. For such~$x$, we can integrate by parts:
\begin{equation*}
(-\Delta)^{\frac12}\varphi(x) = I_1[-\Delta\varphi](x) = -c_1\int\limits_{\mathbb{R}^d}\Delta\varphi(x-t)|t|^{1-d} = -c_1'\int\limits_{\mathbb{R}^d}\varphi(x-t)c|t|^{-1-d}, \quad x \notin \supp \varphi,
\end{equation*}
here~$c_1'$ denotes the numerical constant that arises from the differentiation of the potential. 

Third, we have
\begin{equation}\label{RescalingBound}
\Big|(-\Delta)^{\frac12}\Big[t^{-d}\varphi\Big(\frac{\cdot}{t}\Big)\Big]\Big|(x) \lesssim (t+|x|)^{-d-1},\quad t > 0,
\end{equation}
uniformly with respect to~$t$.

Now let~$\psi$ be an arbitrary~$C_0^{\infty}$ function. By the very definition,
\begin{equation*}
\|(-\Delta)^{\frac12}\varphi\|_{\HS_1} \asymp \Big\|\sup_{t>0}\Big|(-\Delta)^{\frac12}\varphi*t^{-d}\psi\Big(\frac{\cdot}{t}\Big)\Big|\Big\|_{L_1}.
\end{equation*}
Therefore, it suffices to bound the supremum on the right-hand side (as a function of~$x$) by~$(1+|x|)^{-d-1}$. Again, since~$(-\Delta)^{\frac12}\varphi \in L_{\infty}$,
\begin{equation*}
\sup_{t>0}\Big|(-\Delta)^{\frac12}\varphi*t^{-d}\psi\Big(\frac{\cdot}{t}\Big)\Big| \lesssim 1.
\end{equation*} 
So, it suffices to obtain the bound far from the support of~$\varphi$ (say, for the points~$x$ such that~$\dist(x,\supp\varphi) \geq 1$). For them, we can use inequality~\eqref{RescalingBound} (for the function~$\psi$ instead of~$\varphi$):
\begin{equation*}
\Big|(-\Delta)^{\frac12}\varphi*t^{-d}\psi\Big(\frac{\cdot}{t}\Big)\Big|(x) = \Big|\varphi*(-\Delta)^{\frac12}\Big[t^{-d}\psi\Big(\frac{\cdot}{t}\Big)\Big]\Big|(x) \lesssim \Big|\varphi*(t+|\cdot|)^{-d-1}\Big|(x) \leq\big|\varphi*|\cdot|^{-d-1}\big|(x) \lesssim|x|^{-d-1}.
\end{equation*} 
\end{proof}
We note that Lemma~\ref{TestFunctionLemma} is not new. For example, it is a particular case of a more advanced study in~\cite{Str}.

Now fix a hat-function~$\theta$ (i.e. a~$C_0^{\infty}(\mathbb{R}^d)$-function that is non-negative and equals one on the unit ball). Let~$R$ be a positive real number, then~$\theta_R(x) = \theta(\frac{x}{R})$. Using Lemma~\ref{TestFunctionLemma} for~$\varphi = \theta$ and rescaling, we get a corollary.
\begin{Cor}\label{DilationCorollary}
For any~$R > 0$ there exists a real-valued function~$\Theta_R \in \HS_1$ such that
\begin{equation*}
I_1[\Theta_R] \geq \chi_{B_R(0)};\quad \|\Theta_R\|_{\HS_1} \lesssim R^{d-1}
\end{equation*}
uniformly in~$R$.
\end{Cor}
The symbol~$\chi_{\omega}$ denotes the characteristic function of a measurable set~$\omega$;~$B_{r}(z)$ stands for the ball of radius~$r$ centered at~$z$. Specifically, one may take~$\Theta_R = (-\Delta)^{\frac12}\theta_R$. Obviously, one can change the ball centered at the origin for any other ball of the same radius. So, we have proved Theorem~\ref{DorronsoroTheorem} ``for the case where~$f$ is a characteristic function of a ball''. The latter part of the proof is very standard (for example, a similar method leads to the characterization of measures~$\mu$ such that~$\dot{W}_1^1 \hookrightarrow L_{q}(\mu)$, see~\cite{MazjaShaposhnikova}), the idea is to break the function~$f$ into characteristic functions of balls with the control of the~$\dot{W}_1^1$-norm. For that purpose we need the notion of Hausdorff capacity.
\begin{Def}
Let~$\alpha \in [0,d]$. The~$\alpha$-Hausdorff capacity of the set~$\omega \subset \mathbb{R}^d$ is defined by the formula
\begin{equation}\label{CapacityFormula}
H^{\alpha}_{\infty}(\omega) = \inf_{\mathcal{B}} \sum r_j^{\alpha},
\end{equation}
where the infimum is taken over all the coverings~$\mathcal{B}$ of~$\Omega$ by closed balls \textup(and the~$r_j$ are the radii of the balls\textup).
\end{Def}
The Proposition below may be interpreted as ``the case where~$f$ is a characteristic function of a set'' in Theorem~\ref{DorronsoroTheorem}.
\begin{St}\label{DorronsoroForSets}
Let~$\omega$ be an open subset of~$\mathbb{R}^d$. There exists a real-valued function~$\Omega \in \HS_1$ such that
\begin{equation*}
I_1[\Omega] \geq \chi_{\omega};\quad \|\Omega\|_{\HS_1} \lesssim H^{d-1}_{\infty}(\omega).
\end{equation*}
\end{St}
To prove the proposition, one simply considers an almost optimal (in formula~\eqref{CapacityFormula}) covering of~$\omega$ by the balls~$B_{r_j}(x_j)$ and take~$\Omega$ to be~$\sum \Theta_{r_j,x_j}$, where~$\Theta_{r_j,x_j}$ denotes the function~$\Theta_{r_j}$ from Corollary~\ref{DilationCorollary} adjusted to the ball~$B_{r_j}(x_j)$.
\paragraph{Proof of Theorem~\ref{DorronsoroTheorem}.} By using dilations, we may assume that~$f$ is supported in a unit cube, and multiplying it by an appropriate scalar, we may assume that~$\|\nabla f\|_{L_1}=1$. For any~$j \in \mathbb{Z}_+$, define~$\omega_j = \{x \in \mathbb{R}^d\mid f(x) > j\}$. For each~$\omega_j$, we construct a real-valued~$\HS_1$-function~$\Omega_j$ such that
\begin{equation*}
I_1[\Omega_j] \geq \chi_{\omega_j};\quad \|\Omega_j\|_{\HS_1} \lesssim H^{d-1}_{\infty}(\omega_j).
\end{equation*}
Such functions~$\Omega_j$ exist by virtue of Proposition~\ref{DorronsoroForSets}.
Define~$F$ by the formula
\begin{equation*}
F = \sum\limits_{j \geq 0}\Omega_j.
\end{equation*}
Then,
\begin{equation*}
f \leq \sum\limits_{j \geq 0}\chi_{\omega_j} \leq \sum\limits_{j \geq 0} I_1[\Omega_j] = I_1 [F].
\end{equation*}
Moreover,
\begin{equation*}
\begin{aligned}
\|F\|_{\HS_1} \leq \sum\limits_{j=0}^{\infty}\|\Omega_j\|_{\HS_1} \lesssim \sum\limits_{j=0}^{\infty} H^{d-1}_{\infty}(\omega_j) \lesssim 1+ \int\limits_{0}^{\infty} H^{d-1}_{\infty}(\{x\in \mathbb{R}^d\mid f(x) > t\}) \lesssim \\
1+\int\limits_{\mathbb{R}} H^{d-1}(f^{-1}(t)) = 2\|\nabla f\|_{L_1}.
\end{aligned}
\end{equation*}
Here~$H^{d-1}$ denotes the Hausdorff~$(d-1)$-measure. The last but one inequality is an application of Gustin's inequality, Theorem~\ref{GustinTheorem} (note that, by Sard's theorem, almost all sets~$\{x\in \mathbb{R}^d\mid f(x) > t\}$ have smooth boundary), the last one is the coarea formula.\qed

\section{Embeddings for vector fields}\label{sGeneral}

We present a general statement that lies behind Theorem~\ref{DorronsoroTheorem}. In what follows, let~$E$ and~$F$ be two finite dimensional vector spaces over~$\mathbb{C}$. Consider a function~$A: \mathbb{R}^d\times E \mapsto F$ that is a homogeneous polynomial of order~$m$ with respect to the first variable and a linear transformation with respect to the second one. In such a case,~$A$ generates the differential operator that maps~$E$-valued vector fields on~$\mathbb{R}^d$ to~$F$-valued vector fields by the rule
\begin{equation*}
A(\partial)f = \mathcal{F}^{-1}\Big[A\big(i \xi,\mathcal{F}[f](\xi)\big)\Big],\quad f:\mathbb{R}^d\to E,
\end{equation*}
the symbol~$\mathcal{F}$ denotes the Fourier transform. Surely, the field~$f$ must be sufficiently smooth (e.g. belong to the Schwartz class). For example, the differential operator~$\nabla$ corresponds to the function~$A_{\nabla}$ given by the formula
\begin{equation*}
\mathbb{R}^d \ni A_{\nabla}(\xi, e) = \xi e, \quad e\in \mathbb{R},\xi \in \mathbb{R}^d.  
\end{equation*}
\begin{Th}[Van Schaftingen's theorem,~\cite{vanSchaftingen2}]\label{canSchaftingenTheorem}
The inequality
\begin{equation*}
\|\nabla^{m-1} f\|_{L_{\frac{d}{d-1}}} \lesssim \|A(\partial)f\|_{L_1}
\end{equation*}
holds if and only if the polynomial~$A$ is elliptic \textup(i.e.~$A(\xi,e) = 0$ if and only if~$e=0$ or~$\xi = 0$\textup) and cancelling\textup, i.e.
\begin{equation*}
\cap_{\xi\in\mathbb{R}^d \setminus \{0\}}A(\xi,E) = \{0\}.
\end{equation*}
\end{Th}
Surprisingly, there is no result that is similar to Corollary~\ref{LorentzSobolevEmbedding} (this is an open problem whether a similar theorem can be stated with the Lebesgue norm~$L_{\frac{d}{d-1}}$ replaced by the Lorentz norm~$L_{\frac{d}{d-1},1}$; see the recent survey~\cite{vS2}) in such a general setting. However, we can say something. We need one more definition.
\begin{Def}\label{FractionalMaximalFunction}
Let~$a \in [0,d)$. If~$f$ is a locally summable function on~$\mathbb{R}^d$ \textup(or a measure of locally bounded variation\textup)\textup, then the fractional maximal operator of order~$a$ acts on it by the formula
\begin{equation*}
M_a[f](x) = \sup_{r > 0}\;r^{a-d}\!\!\!\int\limits_{|x-y| \leq r} |f|(y)\,dy.
\end{equation*}
\end{Def}
For a measure, the integral over a ball is replaced by the total variation over the same ball. In particular,~$M_0$ is the usual Hardy--Littlewood maximal operator.
\begin{Th}\label{DorronsoroGeneralization}
Let~$A$ be as above\textup, let~$l$ be any non-zero element of~$E^*$\textup, let~$j=1,2,\ldots,d$. The two statements below are equivalent.
\textup{\begin{enumerate}
\item \emph{For any smooth compactly supported vector field~$\varphi$ there exists a real-valued function~$\Phi$ such that
\begin{equation*}
I_1[\Phi] \geq \Re\scalprod{\partial_j^{m-1}\varphi}{l};\quad \|\Phi\|_{\HS_1} \lesssim \|A(\partial)\varphi\|_{L_1}.
\end{equation*}
}
\item \emph{For any smooth compactly supported vector field~$\varphi$ and every non-negative Borel measure~$\mu$
\begin{equation*}
\Re\int\limits_{\mathbb{R}^d}\scalprod{\partial_j^{m-1}\varphi}{l}\,d\mu \lesssim\|A(\partial)\varphi\|_{L_1}\|M_1[\mu]\|_{L_{\infty}}.
\end{equation*}}
\end{enumerate}}
\end{Th}
\begin{proof}
We are going to apply Ky Fan's minimax theorem, Theorem~\ref{KyFanTheorem}. Let~$X$ be the unit ball of the~$\BMO$ space, this set is convex and compact (in the topology~$\sigma(\BMO,\HS_1)$, we use the fact that~$\BMO$ is dual to~$\HS_1$). Let~$Y$ be given by the formula
\begin{equation*}
Y = \{g \in \HS_1(\mathbb{R}^d)\mid\,g\;\hbox{is real-valued,}\; I_1[g] \geq \Re\scalprod{\partial_j^{m-1}\varphi}{l}\}.
\end{equation*}
The function~$L: X\times Y \to \mathbb{R}$ is defined as follows:
\begin{equation*}
L(f,g) = \Re\scalprod{f}{g}.
\end{equation*} 
This function is continuous and bilinear. So, by Theorem~\ref{KyFanTheorem} (we have interchanged the minimum and maximum, we can do this by applying the theorem to the function~$-L$, because we are working with a bilinear function~$L$),
\begin{equation*}
\max_{f\in X}\min_{g \in Y} \Re\scalprod{f}{g} = \min_{g \in Y}\max_{f\in X} \Re\scalprod{f}{g}.
\end{equation*}
The value on the right-hand side is (by the~$\HS_1$-$\BMO$ duality) 
\begin{equation*}
\min \{\|g\|_{\HS_1}\mid\,g\;\hbox{is real-valued,}\; I_1[g] \geq \Re\scalprod{\partial_j^{m-1}\varphi}{l}\}.
\end{equation*}
So, the first of the two statements listed in Theorem~\ref{DorronsoroGeneralization} is equivalent to the inequality
\begin{equation*}
\max_{f\in X}\min_{g \in Y} \Re\scalprod{f}{g} \lesssim \|A(\partial)\varphi\|_{L_1}.
\end{equation*}
Let us calculate the value on the left-hand side (we fix some function~$f$ for a while):
\begin{equation*}
\Re\scalprod{f}{g} = \scalprod{I_1[g]}{\Re(-\Delta)^{\frac12}[f]}.
\end{equation*}
This formula is meaningful, for example, when~$I_1[g] \in C_{0}^{\infty}$. If~$\Re(-\Delta)^{\frac12}[f]$ is not a non-negative distribution, then~$\min_{g \in Y}\Re\scalprod{f}{g}$ equals~$-\infty$. Indeed, this follows from Lemma~\ref{TestFunctionLemma}: if~$\scalprod{\phi}{\Re(-\Delta)^{\frac12}[f]} < 0$ for some non-negative~$C_0^{\infty}$-function~$\phi$, then the value
\begin{equation*}
\Big\langle\Re\scalprod{\partial_j^{m-1}\varphi}{l} + \lambda \phi, \Re(-\Delta)^{\frac12}[f]\Big\rangle
\end{equation*}
can be as small as we want when~$\lambda$ is big (and, by Lemma~\ref{TestFunctionLemma},~$\Re\scalprod{\partial_j^{m-1}\varphi}{l} + \lambda \phi = I_1[g_{\lambda}]$ for some~$g_{\lambda} \in Y$). By the Schwartz theorem, non-negative distributions are (real-valued non-negative) measures of locally bounded variation. But if~$\Re(-\Delta)^{\frac12}[f] = \mu_f$ is a measure, then 
\begin{equation*}
\min_{g \in Y} \Re\scalprod{f}{g} = \min_{g \in Y}\scalprod{I_1[g]}{\Re(-\Delta)^{\frac12}[f]} = \int\Re\scalprod{\partial_j^{m-1}\varphi}{l}\,d\mu_f,
\end{equation*}
where~$\mu_f$ is a non-negative measure of locally bounded variation such that~$\|I_1[\mu_f]\|_{\BMO} \leq 1$; this formula is obvious for the case~$I_1[g] \in C_0^{\infty}$, in the other cases it may be obtained by approximation. Thus, by Adams's theorem (Theorem~\ref{AdamsTheorem}),
\begin{equation*}
\max_{f\in X}\min_{g \in Y} \Re\scalprod{f}{g} \asymp \max\Big(\big\{\Re\scalprod{\int\partial_j^{m-1}\varphi\,d\mu}{l}\,\big|\; \mu \;\hbox{is a non-negative measure such that}\; \|M_1[\mu]\|_{L_{\infty}} \leq 1\big\}\Big).
\end{equation*}
So, the second statement of Theorem~\ref{DorronsoroGeneralization} is equivalent to the inequality
\begin{equation*}
\max_{f\in X}\min_{g \in Y} \Re\scalprod{f}{g} \lesssim \|A(\partial)\varphi\|_{L_1}.
\end{equation*}
\end{proof}
Theorem~\ref{DorronsoroGeneralization} shows that statements in the spirit of Dorronsoro's theorem are, in a sense, equivalent to the fact that the class of measures~$\mu$ such that
\begin{equation*}
\|\nabla^{m-1}f\|_{L_{1}(\mu)} \lesssim \|A(\partial)[f]\|_{L_1}
\end{equation*}
does not depend on the operator~$A$.

\section{Our tools}
\begin{Th}[Gustin's boxing inequality,~\cite{Gustin}]\label{GustinTheorem}
Let~$\omega$ be an open bounded subset of~$\mathbb{R}^d$ with smooth boundary. Then\textup,
\begin{equation*}
H^{d-1}_{\infty}(\omega) \lesssim H^{d-1}(\partial\omega).
\end{equation*}
\end{Th}

\begin{Th}[Ky Fan's minimax theorem]\label{KyFanTheorem}
Let~$X$ and~$Y$ be convex subsets of linear topological spaces\textup, let~$X$ be compact. If a continuous function~$L:X\times Y \to \mathbb{R}$ is convex with respect to the first variable and concave with respect to the second one\textup, then
\begin{equation*}
\min_{x\in X}\max_{y \in Y} L(x,y) = \max_{y \in Y}\min_{x\in X} L(x,y).
\end{equation*}
\end{Th}
We have stated a simplification of Ky Fan's theorem (for the original version, see the paper~\cite{KyFan}\footnote{Our simplification may be an earlier version of the minimax theorem, it may be a result some other mathematician.}).
\begin{Th}[Adams's theorem]\label{AdamsTheorem}
Let~$a \in (0,d)$ be a fixed number. Then\textup,
\begin{equation*}
\|I_a[f]\|_{\BMO} \lesssim \|M_a[f]\|_{L_{\infty}}.
\end{equation*}
If~$f$ is non-negative and~$\int\limits_{\mathbb{R}^d}(1 + |x|)^{-a-d}I_a[f](x)\,dx < \infty$\textup, then
\begin{equation*}
\|M_a[f]\|_{L_{\infty}} \lesssim \|I_a[f]\|_{\BMO} .
\end{equation*}
\end{Th}
This theorem was proved in the paper~\cite{Adams75}.

Dmitriy M. Stolyarov

Institute for Mathematics, Polish Academy of Sciences, Warsaw;

P. L. Chebyshev Research Laboratory, St. Petersburg State University;

St. Petersburg Department of Steklov Mathematical Institute, Russian Academy of Sciences (PDMI RAS).

\medskip

dms at pdmi dot ras dot ru, dstolyarov at impan dot pl.

\medskip

http://www.chebyshev.spb.ru/DmitriyStolyarov.
\end{document}